\documentclass{amsart}
\pdfoutput=1
\usepackage{amsmath}
\usepackage{amssymb}
\usepackage{lmodern}
\usepackage{tikz-cd}
\usepackage{microtype}

\newtheorem{theorem}{Theorem}[section]
\newtheorem*{utheorem}{Theorem}
\newtheorem{lemma}[theorem]{Lemma}
\newtheorem{corollary}[theorem]{Corollary}
\theoremstyle{remark}
\newtheorem{remark}[theorem]{Remark}

\usepackage[lite]{amsrefs}

\newcommand{\red}{\mathrm r}
\newcommand{\norm}[1]{\left\|#1\right\|} 
\newcommand{\Mloc}{M_{\mathrm{loc}}}

\usepackage[pdftitle={Inductive Limits of Noncommutative Cartan Inclusions},
pdfsubject={Mathematics}]{hyperref}

\title{Inductive Limits of Noncommutative Cartan Inclusions}
\author{Ralf Meyer}
\address{Ralf Meyer, Mathematisches Institut, Universität Göttingen, Bunsenstraße 3-5, 37073 Göttingen, Germany}
\email{rmeyer2@uni-goettingen.de}

\author{Ali I. Raad}
\address{Ali I. Raad, Department of Mathematics, KU Leuven, 200B Celestijnenlaan, 3001 Leuven, Belgium}
\email{ali.imadraad@kuleuven.be}

\author{Jonathan Taylor}
\address{Jonathan Taylor, Mathematisches Institut, Universität Göttingen, Bunsenstraße 3-5, 37073 Göttingen, Germany}
\email{jonathan.taylor@mathematik.uni-goettingen.de}
\subjclass[2010]{Primary 46L05.}

\thanks{The second author was supported by the Internal KU Leuven
  BOF project C14/19/088 and project G085020N funded by the Research
  Foundation Flanders (FWO).  The third author was supported by a stipend from the German Academic Exchange Service (DAAD) Funding Program 57450037}
\begin{document}
\begin{abstract}
  We prove that an inductive limit of aperiodic noncommutative
  Cartan inclusions is a noncommutative Cartan inclusion whenever
  the connecting maps are injective, preserve normalisers and
  entwine conditional expectations.  We show that under the
  additional assumption that the inductive limit Cartan subalgebra
  is either essentially separable, essentially simple or essentially
  of Type I we get an aperiodic inclusion in the
  limit.  Consequently, we subsume the case where the building block
  Cartan subalgebras are commutative and provide a proof of a theorem of Xin Li without passing to twisted \'etale groupoids.
\end{abstract}
\maketitle

\section{Introduction}

The theory of Cartan subalgebras for operator algebras has been
prevalent since Murray and von Neumann's construction of the
algebras $L^\infty(X,\mu) \rtimes G$ arising from nonsingular group
actions on a measurable space $G \curvearrowright (X,\mu)$.
Although not termed a Cartan subalgebra then, the distinguished
subalgebra $L^\infty(X,\mu) \subseteq L^\infty(X,\mu) \rtimes G$ is
indeed a Cartan subalgebra (a regular inclusion of a masa admitting
a faithful normal conditional expectation).  The abstract definition
for a Cartan subalgebra in a von Neumann algebra was later given by
Vershik (see~\cite{Vershik:Decompositions}).  A characterisation of such inclusions
was provided by Feldman and Moore in~\cite{Feldman-Moore:Ergodic_II}; the Cartan
subalgebras are certain subalgebras of von Neumann algebras
constructed from measured countable equivalence relations.

The theory of Cartan subalgebras in the setting of C$^*$-algebras
was thereafter developed by Kumjian and Renault (\cite{Kumjian:Diagonals}
and~\cite{Renault:Cartan.Subalgebras}, respectively) and characterised as the inclusion
of the $C_0$-functions on the unit space of an étale effective
twisted groupoid inside the reduced C$^*$-algebra of the twisted
groupoid.  Shortly after this characterisation, Exel defined a
notion of a noncommutative Cartan subalgebra in~\cite{Exel:noncomm.cartan}, where
the condition of being maximally commutative was replaced with the
condition of having trivial virtual commutants.  In the commutative
case, this condition is exactly the one of being maximally
commutative.  Exel also showed that every noncommutative Cartan
inclusion was an inclusion inside the reduced cross-sectional
C$^*$-algebra of a Fell bundle over an inverse semigroup, where the
Cartan subalgebra corresponds to the reduction of the semigroup to
the lattice of idempotents.  Kwa\'sniewski and the first author
improved upon Exel's theory in~\cite{Kwasniewski-Meyer:Cartan} by completely
characterising the types of actions that yield noncommutative Cartan
inclusions.

In recent years, the commutative setting for Cartan subalgebras has
attracted a lot of attention.  They are related to topological
dynamical systems via continuous orbit equivalence, to geometric
group theory via quasi-isometry (see~\cite{Li:Dynamic_quasi-isometry}), and also to
the classification programme for C$^*$-algebras, which aims at
classifying a certain class of `well-behaved' C$^*$-algebras by an
invariant consisting of $K$-theoretic and tracial data.  A
breakthrough result by Li in~\cite{Li:Classifiable_Cartan} shows that every such
classifiable C$^*$-algebra has a Cartan subalgebra.

In the same work, Li provides sufficient conditions on connecting
maps of an inductive system of Cartan inclusions which guarantee
that the inductive limit is a Cartan inclusion.  One requires the
connecting maps to be injective, map Cartan subalgebra to Cartan
subalgebra, normalisers to normalisers and entwine the faithful
conditional expectations (see Theorem 1.10 in~\cite{Li:Classifiable_Cartan}).  In many
classes of examples, it is significantly easier to check that such
conditions hold rather than attempting to find a Cartan subalgebra
in the inductive limit directly (for instance, in AF-algebras, where
connecting maps are well-understood).  In fact, Li and the second author have used these conditions to construct inductive limit
Cartan subalgebras in many classes of AH-algebras, many of which are
not classifiable (see \cite{Li-Raad:Diagonals_AH}).  The main result of this
article generalises Li's theorem to inductive systems of
noncommutative Cartan inclusions:

\begin{utheorem}[see Theorem~\ref{theo-MainTheorem}] Given an
  inductive system of aperiodic noncommutative Cartan inclusions
  where the connecting maps are injective and map Cartan subalgebra
  to Cartan subalgebra, normalisers to normalisers and entwine the
  conditional expectations, the inductive limit is a
  noncommutative Cartan inclusion. If the limit Cartan subalgebra
  has an essential ideal that is separable, simple or of Type~I,
  then the limit inclusion is aperiodic.
\end{utheorem}

As a consequence, we get a sufficient condition on the level of
connecting maps that guarantees that the inductive limit is a
canonical noncommutative Cartan inclusion.  In the commutative
case, it turned out to be much easier to check such conditions
rather than work directly with the inductive limit.

Our result also subsumes the result on Cartan subalgebras of
inductive limits in Theorem 1.10 in~\cite{Li:Classifiable_Cartan} because all
commutative Cartan inclusions are aperiodic (see
Remark~\ref{rem-SubsumeCommutative}).  Our proof of this theorem does not rely on passing to \'etale twisted
groupoids.

The paper is organised as follows.  Section~\ref{sec:prelim} will
briefly provide preliminaries on noncommutative and aperiodic Cartan
inclusions, and will set up the standing assumptions on our
inductive system.  Section~\ref{sec:indlim} will present the proof
of our main theorem.

\section{Preliminaries}
\label{sec:prelim}

Throughout this article, we will consider an inductive system of
C$^*$-algebras
\[
  \begin{tikzcd}
    A_1 \arrow[r,"\phi_1"] & A_2 \arrow[r,"\phi_2"] & A_3
    \arrow[r,"\phi_3"] \arrow[rr,bend left, dashed, "\mu_n"]& \ldots
    & A \\
    C_1 \arrow[r,"\phi_1"] \arrow[u,hookrightarrow] & C_2
    \arrow[u,hookrightarrow] \arrow[r,"\phi_2"] & C_3
    \arrow[u,hookrightarrow] \arrow[r,"\phi_3"] \arrow[rr,bend
    right, dashed, "\mu_n"] & \ldots & C,
    \arrow[u,hookrightarrow,dashed]
  \end{tikzcd}
\]
where the vertical arrows are set inclusions.  The building block
inclusions $C_n \subseteq A_n$ are assumed to be nondegenerate
noncommutative Cartan inclusions as in \cite{Exel:noncomm.cartan}*{Definition~2.1}.
This means the following:
\begin{enumerate}
\item $C_n$ is a C$^*$-subalgebra of $A_n$ that is regular, that is,
  the set of normalisers
  $N_{A_n}(C_n) := \{n \in A_n : n^*C_nn \subseteq C_n, nC_nn^*\subseteq
  C_n\}$ generates~$A_n$ as a C$^*$-algebra;
\item there is a faithful conditional expectation
  $P_n\colon A_n \twoheadrightarrow C_n$;
\item $C_n$ contains an approximate unit for~$A_n$;
\item virtual commutants of~$C_n$ in~$A_n$ are trivial.
\end{enumerate}
Here a \emph{virtual commutant} of~$C_n$ in~$A_n$ is a bounded
$C_n$-bimodule map $\varphi\colon J_n \rightarrow A_n$ for some
closed two-sided ideal $J_n\subseteq C_n$.  A virtual commutant
of~$C_n$ in~$A_n$ is trivial if the image of the map lies in~$C_n$.
For further details consult~\cite{Exel:noncomm.cartan}.

We further assume that the inclusion $C_n\subseteq A_n$ is
aperiodic.  This means that the Banach $C_n$\nobreakdash-bimodule
$X_n=A_n/C_n$ is an aperiodic $C_n$-bimodule, that is, for each
$x \in X_n$ and each nonzero hereditary subalgebra $D\subseteq C_n$
and $\epsilon>0$, there is a positive element $d\in D$ with
\(\norm{d}=1\) and $\norm{dxd}< \epsilon$ (see
\cite{Kwasniewski-Meyer:Cartan}*{Definition~6.1}).

We assume the connecting maps $\{\phi_n\}_{n \in \mathbb{N}}$ to be
nondegenerate and injective $^*$\nobreakdash-homomorphisms.  This
gives rise to a nondegenerate inclusion of C$^*$-algebras
$C\subseteq A$ with nondegenerate and injective structure
$^*$\nobreakdash-homomorphisms $\{\mu_n\}_{n\in \mathbb{N}}$.  To
simplify notation, we may identify building block algebras with their images under connecting maps, so that we may consider the maps~\(\phi_n\)
and~\(\mu_n\) as inclusions of C$^*$-subalgebras and drop them
from our notation.

We further place assumptions on the connecting maps that are
analogous to those in Theorem 1.10 in~\cite{Li:Classifiable_Cartan}, namely:
\begin{enumerate}
\item they map normalisers to normalisers, that is,
  $\phi_n(N_{A_n}(C_n))\subseteq N_{A_{n+1}}(C_{n+1})$;
\item they entwine the conditional expectations, that is,
  $P_{n+1}\circ \phi_n=\phi_n\circ P_n$.
\end{enumerate}

For an inclusion of C$^*$-algebras $\mathcal{C}\subset\mathcal{A}$  we will call a subset $M\subset\mathcal{A}$ a \emph{slice} for the inclusion if $M$ is a closed linear
subspace of $N_{\mathcal{A}}(\mathcal{C})$ that is also a $\mathcal{C}$-bimodule. For
$n\in \mathbb{N}$, let~$S_n$ be the inverse semigroup of slices for
the inclusion \(C_n \subseteq A_n\); its multiplication is defined
by taking the closure of the linear span of the algebraic
multiplication, and the inverse by taking the involution * (see
Section 10 in \cite{Exel:noncomm.cartan}).  For subsets $A$ and $B$ of a
C$^*$-algebra, we will denote the aforementioned multiplication by
$A\cdot B=\overline{\mathrm{span}(AB)}$.  For an element
$m\in N_{A_n}(C_n)$, $C_n\cdot \{m\}\cdot C_n$ is a slice (see
\cite{Exel:noncomm.cartan}*{Proposition~10.5}). Every slice is contained in a sum of slices of this form. Indeed, by the
Cohen--Hewitt Factorisation Theorem (\cite{Hewitt-Ross:Abstract_harmonic_analysisII}*{Theorem~32.22}) we
can write every slice~$M$ as $C_nMC_n$, which is contained in
$\sum_{m \in M}C_n\cdot \{m\}\cdot C_n$.

An inductive system of slices $\mathcal{F}=\{M_n,\phi_n\}_{n \in \mathbb{N}}$
consists of slices $M_n \in S_n$ with
$\phi_{n}(M_n)\subseteq M_{n+1}$.  This system of slices gives rise
to a limit slice
$F_{\mathcal{F}}=\overline{\bigcup_{n\in\mathbb{N}}\mu_n(M_n)}$.

Define $P\colon A \rightarrow C$ as the (unique) extension of
$P_0\colon \bigcup_{n}\mu_n(A_n) \rightarrow
\bigcup_{n}\mu_n(C_n)$ defined by
$P_0(\mu_n(a))=\mu_n(P_n(a))$, $a \in A_n$.  Since the connecting
maps entwine conditional expectations and each~$P_n$ is contractive,
the map~$P_0$ is well-defined and contractive.  A conditional
expectation $Q\colon B\rightarrow D$ is \emph{faithful} if no
nonzero positive element of~$B$ is mapped to zero.  It is
\emph{almost faithful} if $Q(x^*b^*bx)=0$ for all $x\in B$ and some
$b \in B$ implies $b=0$.  It is \emph{symmetric} if $Q(b^*b)=0$ is
equivalent to $Q(bb^*)=0$.

For an inclusion of C$^*$-algebras $\mathcal{C}\subset\mathcal{A}$ a \emph{generalized expectation} is a completely positive contractive map $E:\mathcal{A}\rightarrow \tilde{\mathcal{C}}$ such that $E\vert_\mathcal{C}=\mathrm{id}$, where $\mathcal{C}\subset\tilde{\mathcal{C}}$ is an inclusion of C$^*$-algebras. If $\tilde{\mathcal{C}}=I(\mathcal{C})$ (where $I(\mathcal{C})$ is Hamana's injective hull, see \cite{Hamana:Injective-Envelope-Cstar}) then $E$ is called a \emph{pseudo-expectation}. For details, consult \cite[Section~3]{MR4485960}.

Let~$\mathcal{P}$ be a property for C$^*$-algebras (for example,
separable).  We call a C$^*$-algebra
\emph{essentially~$\mathcal{P}$} if it contains an essential ideal
with property~$\mathcal{P}$.  Some results in~\cite{MR4485960} can only be
applied if the C$^*$-algebra~\(A\) is \emph{essentially separable},
\emph{essentially simple}, or \emph{essentially of Type~I}.

\section{Inductive limits of noncommutative Cartan inclusions}
\label{sec:indlim}

In this section we prove our main result.  Unless otherwise stated,
we assume throughout that we are in the setting given in the
preliminaries.

\begin{lemma}
  \label{lem-indLimIsRegWithAlmostFaithfulCondExp}
  The inclusion $C \subseteq A$ is regular, and $P$ is an almost
  faithful conditional expectation.
\end{lemma}

\begin{proof}
  We first show that $C \subseteq A$ is a regular inclusion.  Let
  $m \in N_{A_n}(C_n)$.  We are going to prove that
  \(\mu_n(m)\in A\) normalises~\(C\).  Let $c\in C$.  There are
  $c_j \in C_j$ with \(\lim c_j = c\).  Then
  \[
    \mu_n(m^*)c\mu_n(m)
    = \lim_{j\to\infty} \mu_n(m^*)\mu_j(c_j)\mu_n(m)
    = \lim_{j\to\infty} \mu_j\bigl(\phi_{jn}(m^*) c_j \phi_{jn}(m)\bigr).
  \]
  Since~$\phi_{jn}(m)$ is a normaliser in~$A_j$ by assumption,
  \(\phi_{jn}(m^*) c_j \phi_{jn}(m) \in C_j\) and so the limit
  belong to~$C$.  This finishes the proof that \(\mu_n(m)\in A\)
  normalises~\(C\).  Therefore, the C$^*$-algebra generated by
  $N_A(C)$ contains the C$^*$-algebra generated by all the
  $\mu_n(A_n)$, which is all of~$A$.

  Next, we show that~$P$ is almost faithful.  Equivalently,
  $\mathcal{N}:=\{a \in A: P(b^*a^*ab)=0 \text{ for all } b\in A\}$
  vanishes.  \cite{Kwasniewski:Exel_crossed}*{Proposition 2.2} implies that~$\mathcal{N}$
  is the largest ideal of~$A$ contained in $\operatorname{ker}(P)$.
  We know that $\mathcal{N}_n=\mathcal{N} \cap \mu_n(A_n)$ is an
  ideal of $\mu_n(A_n)$.  Thus
  $\mathcal{N}=\overline{\bigcup_n\mathcal{N}_n}$ (see
  \cite{Davidson:Cstar_example}*{Lemma III.4.1}).  Since~$P$ vanishes
  on~$\mathcal{N}_n$, it follows that~$P_n$ vanishes on
  $\mu_n^{-1}(\mathcal{N}_{n})$.  Since~$P_n$ is almost faithful,
  even faithful, this forces $\mathcal{N}_n =\{0\}$.  Hence
  $\mathcal{N}=\{0\}$ as desired.
\end{proof}

\begin{lemma}
  Let $\mathcal{F}=\{M_n,\phi_n\}_{n \in \mathbb{N}}$ be an inductive system of
  slices with limit~$F$.  Then~$F$ is a slice for the inclusion
  $C\subseteq A$.
\end{lemma}

\begin{proof}
  We know that~$F$ normalises~$C$ because each~$M_n$
  normalises~$C_n$.  It is clear that~\(F\) is a closed linear
  subspace.  We claim that $FC\subseteq F$.  Fix $c_k \in C_k$ and
  consider $f=\lim_i \mu_i(m_i) \in F$.  Then
  $fc_k=\lim_i\mu_i(m_i\phi_{ik}(c_k))$ which is a limit of elements
  in $\mu_i(M_i)$ and hence belongs to~$F$.  As~$F$ is closed it
  follows that $FC\subseteq F$.  A similar proof shows
  $CF\subseteq F$.
\end{proof}

We now let~$S$ be the collection of all limits of inductive systems
$\mathcal{F}=\{M_n,\phi_n\}_{n \in \mathbb{N}}$.

\begin{lemma}
  \label{lem-InductiveLimitsOfSlicesIsSemigroup}
  The collection~$S$ is an inverse semigroup of slices, under the
  multiplication~$\cdot$ and the inverse~$^*$.
\end{lemma}

\begin{proof}
  Let $F_1$ and $F_2$ be the limits of inductive systems of slices
  $\mathcal{F}_1=\{M_n,\phi_n\}_{n \in \mathbb{N}}$ and
  $\mathcal{F}_2=\{N_n,\phi_n\}_{n \in \mathbb{N}}$, respectively.  We first show
  that $F_1 \cdot F_2,$ is the limit of the inductive system of
  slices $\{M_n \cdot N_n,\phi_n\}_{n \in \mathbb{N}}$.  First,
  $\bigcup_n\mu_n(\mathrm{span}(M_nN_n))$ is dense in
  $\bigcup_n\mu_n(M_n \cdot N_n)$, and
  $\bigcup_n\mu_n(\mathrm{span}(M_nN_n)) =
  \mathrm{span}((\bigcup_n\mu_n(M_n))(\bigcup_n\mu_n(N_n)))$.  Then
  \begin{multline}\label{eq:multS}
    F_1 \cdot F_2
    =
    \overline{\mathrm{span}(\overline{(\bigcup_n\mu_n(M_n))}\overline{(\bigcup_n\mu_n(N_n))})}
    \\= \overline{\bigcup_n\mu_n(\mathrm{span}(M_nN_n))}
    = \overline{\bigcup_n\mu_n(M_n \cdot N_n)}.
  \end{multline}
  Let us now show that the involution~$^*$ of the C$^*$-algebra $A$ acts as a generalized
  inverse. Let $F\in S$ be the limit of the system $\{M_n,\phi_n\}$, so that $F=\overline{\cup_{n\in \mathbb{N}}\mu_n(M_n)}$. Now note that by continuity of the involution $^*$ we have $F^*=\overline{\cup_{n\in \mathbb{N}}\mu_n(M_n^*)}$ . Hence  by \eqref{eq:multS} it follows that $$F\cdot F^*\cdot F=\overline{\bigcup\limits_{n}\mu_n(M_n\cdot M_n^*\cdot M_n)}=\overline{\bigcup\limits_{n}\mu_n(M_n)}=F.$$ Similarly one can show that $F^*\cdot F\cdot F^*=F^*.$ To obtain uniqueness of the generalized inverse, note that if $F_1, F_2\in S$
  are idempotents, then they are ideals of~$C$ and hence commute.
  Now \cite{Lawson:InverseSemigroups}*{Theorem 3 in Chapter 1.1} gives uniqueness of the generalized inverse and hence \(S\) is an inverse
  semigroup.
\end{proof}

\begin{lemma}
  \label{lem-LineaSpanBasisSlices}
  The linear span of elements in~$S$ is dense in~$A$.
\end{lemma}

\begin{proof}
  Any $m\in N_{A_n}(C_n)$ is contained in the slice
  $C_n \cdot \{m\}\cdot C_n$.  Thus the limit of the inductive
  system of slices
  $\{C_k \cdot \{\phi_{kn}(m)\}\cdot C_k, \phi_k\}_{k \ge n}$
  contains $\mu_n(m)$.  Hence the linear span of elements of~$S$
  contains the linear span of $\bigcup_n\mu_n(N_{A_n}(C_n))$.  The
  latter is dense in~$A$ because each inclusion
  \(C_n \subseteq A_n\) is regular.
\end{proof}

\begin{remark}
  Lemma~\ref{lem-LineaSpanBasisSlices} implies that~$S$ is a
  saturated grading for~$A$ (see \cite{Kwasniewski-Meyer:Cartan}*{Definition~2.1}).  Then
  \cite{Kwasniewski-Meyer:Cartan}*{Remark 2.8} shows that there is a canonical surjective
  $^*$\nobreakdash-homomorphism $U\colon C \rtimes S \rightarrow A$.
\end{remark}

\begin{lemma}
  \label{lem-inductiveLimExpAgreesWithCanon}
  Let $U\colon C\rtimes S\to A$ be the universal surjective
  $^*$-homomorphism and let $EL\colon C\rtimes S\to \Mloc(C)$ be the
  canonical generalized conditional expectation defined in
  \cite{Kwasniewski-Meyer:Cartan}*{Proposition~2.9}.  Then \(EL = P\circ U\), and
  so~\(EL\) takes values in $C\subseteq \Mloc(C)$.
\end{lemma}

\begin{proof}
  By Lemma~\ref{lem-LineaSpanBasisSlices}, elements of~$S$ span a
  dense subset of~$A$.  Since conditional expectations are linear
  and continuous, it suffices to consider restrictions to building
  blocks of inductive systems of slices.  Consider such a building
  block~$M_n$.  It suffices to show that $EL(k)=P_n(k)$ for all
  $k\in M_n$.  On $M_n\cap C_n$ both $P_n$ and~$EL$ restrict to the
  identity map as this is contained in~$C_n$.  The expectation~$EL$
  is zero on the complement $M_n\cdot (M_n\cap C_n)^\perp$ by
  construction (see \cite{Buss-Exel-Meyer:Reduced}*{Lemma~4.5}).  To see that the
  expectation~$P_n$ is zero on $M_n\cdot (M_n\cap C_n)^\perp$, note
  that $P_n$ preserves slices by \cite{Kwasniewski-Meyer:Cartan}*{Lemma~4.10}.  Then
  $P_n(M_n\cdot (M_n\cap C_n)^\perp)\subseteq P(M_n)\cdot (M_n\cap
  C_n)^\perp\subseteq (M_n\cap C_n)\cdot (M_n\cap C_n)^\perp=\{0\}$.
  Since $C_n\subseteq A_n$ is a Cartan inclusion, the slice~$M_n$
  decomposes as $M_n=M_n\cap C_n\oplus M_n\cdot (M_n\cap C_n)^\perp$
  by \cite{Kwasniewski-Meyer:Cartan}*{Proposition~2.17}.  So $P_n=EL$ on each slice.  Then
  $P$ and~$EL$ agree on the inductive limit slices belonging to~$S$.
  Thus \(EL = P\circ U\).
\end{proof}

\begin{corollary}
  \label{cor-IndLimExpSymmetric}
  The inductive limit expectation~$P$ is symmetric.
\end{corollary}

\begin{proof}
  Lemma~\ref{lem-inductiveLimExpAgreesWithCanon} says that
  $EL=P\circ U$.  Fix $a\in A$ with $P(a^*a)=0$.  There is
  $a_0\in C\rtimes S$ with $U(a_0)=a$.  By \cite{Kwasniewski-Meyer:Essential}*{Theorem~4.11},
  the expectation~$EL$ is symmetric.  So
  $EL(a_0^*a_0)=P(U(a_0)^*U(a_0))=P(a^*a)=0$ implies
  $EL(a_0a_0^*)=0$.  Then
  $P(aa^*)=P(U(a_0)U(a_0)^*)=EL(a_0a_0^*)=0$.  This shows that~$P$
  is symmetric.
\end{proof}

\begin{corollary}
  \label{cor-PisFaithful}
  The inductive limit expectation~$P$ is faithful.
\end{corollary}

\begin{proof}
  By Lemma~\ref{lem-indLimIsRegWithAlmostFaithfulCondExp} and
  Corollary~\ref{cor-IndLimExpSymmetric}, $P$ is almost faithful and
  symmetric.  Therefore, it is faithful by
  \cite{Kwasniewski-Meyer:Essential}*{Corollary~3.8}.
\end{proof}

\begin{theorem}
  \label{theo-MainTheorem}
  Let $C_n\subseteq A_n$ be aperiodic noncommutative Cartan
  inclusions.  Let $\phi_n\colon A_n \rightarrow A_{n+1}$ be
  injective and nondegenerate $^*$\nobreakdash-homomorphisms that
  satisfy $\phi_n(C_n)\subseteq C_{n+1}$,
  $\phi(N_{A_n}(C_n))\subseteq N_{A_{n+1}}(C_{n+1})$, and
  $P_{n+1}\circ \phi_n=\phi_n \circ P_n$.  Then
  $C=\varinjlim(C_n,\phi_n) \subseteq A=\varinjlim(A_n,\phi_n)$ is a
  noncommutative Cartan inclusion.  If, in addition, $C$ is
  essentially separable, essentially simple or essentially of
  Type~I, then the inclusion $C\subseteq A$ is aperiodic.
\end{theorem}

\begin{proof}
  We are going to prove by contradiction that~$C$ detects ideals in
  all intermediate C$^*$-algebras $C \subseteq B \subseteq A$. Together with the results in~\cite{MR4485960}, this will imply our claims.  So
  assume that this fails.  Then there is a nonzero ideal
  $I \subseteq B$ with $I \cap C=\{0\}$.
  Fix a nonzero positive element $b \in I$.
  Corollary~\ref{cor-PisFaithful} says that the conditional
  expectation~$P$ is faithful.  So there is $\epsilon>0$ with
  $\norm{P(b)}>\epsilon$.  Then there are $n\in \mathbb{N}$ and a
  positive element $b_n \in A_n$ with $\norm{b_n-b} < \epsilon/3$.

  Let $I(C_n)$ denote Hamana's injective envelope of~$C_n$
  (see~\cite{Hamana:Injective-Envelope-Cstar}) and let $\pi\colon B \rightarrow B/I$ be the
  quotient map.  We have got the following commutative diagram:
  \[
    \begin{tikzcd}
      A_n \arrow[r,hook] & A & \\
      C_n \arrow[d,hook] \arrow[rd,hook] \arrow[u, hook] &
      B\arrow[u,hook] \arrow[r,twoheadrightarrow,"\pi"] &
      B/I \\
      I(C_n) & C. \arrow[u,hook] \arrow[ur,hook]&
    \end{tikzcd}
  \]
  Since~$I(C_n)$ is injective, the identity homomorphism on~$C_n$
  extends to a completely positive contractive map
  $Q\colon B/I \rightarrow I(C_n)$.  Next, \(Q\circ \pi\) extends to
  a completely positive contraction $R\colon A\rightarrow I(C_n)$.
  By construction, $R(b)=Q(\pi(b))=0$ and then
  $\norm{R(b_n)}=\norm{R(b_n-b)}\le\norm{b_n-b} < \epsilon/3$.  The
  reverse triangle inequality implies
  $\epsilon/3 > \norm{P(b_n)-P(b)}\ge
  \bigl|\norm{P(b_n)}-\norm{P(b)}\bigr|$ and then
  $\norm{P(b_n)}>2\epsilon/3>\epsilon/3>\norm{R(b_n)}$.  Hence
  $P\vert_{A_n} \neq R\vert_{A_n}$.  However, both \(P\) and~\(R\)
  are completely positive contractions extending the identity map
  on~\(C_n\), which makes them generalized expectations for the
  inclusion \(C_n \subseteq A_n\).  It is well known that
  \(\Mloc(C_n) \subseteq I(C_n)\).  So \(P\) and~\(R\) produce two
  different pseudo-expectations for the inclusion
  $C_n \subseteq A_n$.  Then \cite{MR4485960}*{Theorem 3.6} implies that
  this inclusion is not aperiodic, in contradiction to our
  assumption.  This finishes the proof that~\(C\) detects ideals in
  any intermediate C$^*$-algebra~\(B\).

  The canonical expectation \(EL\) on \(C\rtimes S\) takes values
  in~\(C\) by Lemma~\ref{lem-inductiveLimExpAgreesWithCanon}.
  Therefore, the reduced and essential crossed products agree for
  the relevant action of~\(S\) on~\(A\).  Since \(EL= P\circ U\)
  and~\(P\) is faithful, it also follows that the canonical
  ${}^*$-homomorphism $U\colon C\rtimes S\to A$ descends to an
  isomorphism $A\cong C \rtimes_\red S$.  If $T\subseteq S$ is an
  inverse subsemigroup that contains all idempotents of~$S$, then
  $C \subseteq C\rtimes_\red T \subseteq C\rtimes_\red S$ is an
  intermediate C$^*$-algebra, and we have shown that~\(C\) detects
  ideals in it.  Now \cite{MR4485960}*{Proposition~6.7} shows that the
  action of~$S$ is purely outer.  Then the inclusion $C\subseteq A$
  is a noncommutative Cartan inclusion by \cite{Kwasniewski-Meyer:Cartan}*{Theorem~4.3}.
  If~$C$ also contains an essential ideal that is separable, simple,
  or of Type~I, then the inclusion $C\subseteq A$ is even aperiodic
  by the conditional implications in \cite{MR4485960}*{Figure~1}.
\end{proof}

\begin{remark}
  \label{rem-SubsumeCommutative}
  The properties of being of Type~I, separable, or simple each pass
  to inductive limits of C$^*$-algebras.  Thus, if all the
  noncommutative Cartan subalgebras~$C_n$ in the building blocks are
  of Type I, separable, or simple, then the inclusion $C\subseteq A$
  is again aperiodic.  This, however, may break down if the building
  block subalgebras are only essentially of Type~I, simple or
  separable, because essential ideals in~\(C_n\) need not survive to
  ideals in~\(C\).

  In particular, if each~\(C_n\) is commutative, then it is of
  Type~I and so there is no difference between aperiodic and purely
  outer actions.  Hence Theorem~\ref{theo-MainTheorem} subsumes the
  setting of \cite{Li:Classifiable_Cartan}*{Theorem~1.10}.  Moreover the argument we give does not rely on passing to \'etale twisted groupoids.
\end{remark}

\begin{remark}
    In the case where each $C_n=C_0(X_n)$ is commutative, consider the Gelfand dual continuous map $f_n:X_{n+1}\to X_n$ inducing $\phi_n|_{C_n}$. 
    If $f_n$ is an open map, then the inclusion $C_n\subseteq C_{n+1}$ induced by $\phi_n|_{C_n}$ is anti-aperiodic, that is, contains no non-zero aperiodic $C_n$-bimodules.
    Indeed, for any non-zero function $g\in C_0(X_{n+1})$ there is an open subset $V\subseteq X_{n+1}$ where $|g(x)|>||g||/2$ for all $x\in V$.
    Since $f_n$ is open we see that $C_0(f_n(V))$ is an ideal in $C_0(X_n)$, and for any $h\in C_0(f_n(V))$ with $||h||=1$ we have $||hgh||\geq\sup_{x\in V} |g(x)||h(f_n(x))|^2>||g||/2$, so can never satisfy Kishimoto's condition.
    
    In this situation we have by \cite{MR4485960}*{Proposition~3.9} that generalised expectations for the inclusion $C_n\subseteq A_n$ taking values in $C_{n+1}$ are unique, since the inclusion $C_n\subseteq A_n$ is aperiodic, and the inclusion $C_n \subseteq C_{n+1}$ is anti-aperiodic.
    The maps $\phi_n\circ P_n$ and $P_{n+1}\circ\phi_n$ are both such generalised expectations, so must then be equal.
    Thus our ${}^*$-homomorphisms $A_n\to A_{n+1}$ entwine conditional expectations automatically if the Gelfand duals of the restrictions $C_n\to C_{n+1}$ are open.
    
    If the Gelfand dual map is not open then inclusions of commutative C$^*$-algebras may not be anti-aperiodic. 
    For example consider $C[0,1]\subseteq C[0,2]$ induced by the continuous function $f:[0,2]\to[0,1]$, defined by
    $$f(t):=\begin{cases}t,&t\leq 1\\
    1,&t>1.
    \end{cases}$$
    The $C[0,1]$-subbimodule $C_0(1,2]$ is annihilated by the essential ideal $C_0[0,1)\subseteq C[0,1]$.
    By \cite{Kwasniewski-Meyer:Essential}*{Lemma~5.12} the bimodule $C_0(1,2]$ is a non-zero aperiodic $C[0,1]$-subbimodule of $C[0,2]$, hence the inclusion is not anti-aperiodic. 
\end{remark}

\end{document}